\documentclass[runningheads]{llncs}
\usepackage{graphicx,amsmath,amssymb}

\usepackage{subfigure}

\newcommand{\dollar}{{\texttt{\$}}}

\newcommand{\obacht}[2][]{}

\newcommand{\ignore}[1]{}

\begin{document}

\title{A Branch-and-Cut Algorithm for the 2-Species Duplication-Loss Phylogeny Problem}
\author{Stefan Canzar\inst{1} \and Sandro Andreotti\inst{2,3}}

\titlerunning{Duplication-Loss Alignment by Cutting Planes}

\institute{McKusick-Nathans Institute of Genetic Medicine, Johns Hopkins University School of Medicine, Baltimore, Maryland, USA \and
Institute of Computer Science, Department of Mathematics and Computer Science, Freie Universit\"at Berlin, Berlin, Germany \and
International Max Planck Research School for Computational Biology and Scientific Computing, Berlin, Germany}

\maketitle

\begin{abstract}
The reconstruction of the history of evolutionary genome-wide events among a set of related organisms is of great biological interest.
A simplified model that captures only content modifying operations was introduced recently.
It allows the small phylogeny problem to be formulated as an alignment problem. 
In this work we present a branch-and-cut algorithm for this so-called duplication-loss alignment problem. Our method clearly outperforms the existing ILP based method by several orders of magnitude.
We define classes of valid inequalities and provide algorithms to separate them efficiently and prove the $NP$-hardness of the duplication-loss alignment problem.
\end{abstract}

\section{Introduction}
In the course of evolution genome-wide changes either (i) rearrange the order of the genes or (ii) modify the content. The former class of changes result from inversions,
transpositions, and translocations, the latter have an effect on the number of gene copies that are either inserted, lost, or duplicated. The reconstruction
of the history of such events among a set of (related) organisms is of great biological interest, since it can help to reveal the genomic basis of phenotypes.
In a recent work \cite{Holloway:2012} the authors study the problem of inferring an ancestral genome, from which two given genomes have evolved by content-modifying operations
of type (ii) only, namely through the duplication and loss of genes. A prominent example of a gene family that is continuously duplicated and lost is transfer RNA
(tRNA) \cite{Rogers2010,Swenson2008,Withers2006}. Since tRNA is an essential element in the translation of RNA into proteins, reconstructing their evolutionary history among species might lead to new insights into the translationary machinery.

The consequences of the evolutionary model that only accounts for the duplication and loss of genes is twofold. First of all, the order of genes is preserved and thus
the problem can be cast into an alignment problem \cite{Holloway:2012}, which is in general favorable from a combinatorial perspective. Secondly, duplications and losses are asymmetric operations and thus an ancestral genome can immediately
be obtained from a duplication/loss scenario.

\subsubsection{Related work and our contribution.}
Holloway \emph{et al.} \cite{Holloway:2012} proposed an approach for the comparison of genomes under the duplication and loss model that is based on an integer linear programming
(ILP) formulation of the problem. While the method in \cite{Holloway:2012} iteratively adds cycle constraints to the ILP, we have developed this idea further into a
cutting plane algorithm. Exploiting insights into the combinatorial structure of the problem we introduce cuts that can be separated efficiently and which 
lead to a branch-and-cut algorithm that outperforms the previous method \cite{Holloway:2012} by several orders of magnitude.
The related problem of labeling a given alignment of two genomes by duplications and losses was recently shown to be $APX$-hard \cite{Dondi:2012}.  We show that the problem
of finding a maximum parsimony ancestral genome of two given genomes is $NP$-hard.

\section{Problem definition}
We start with some basic definitions that are adopted from~\cite{Holloway:2012} 
Given two Genomes $G^a$ and $G^b$ and a set of allowed evolutionary operations $\mathcal O$ a sequence of $n$ operations $\in \mathcal O$ that transforms $G^a$ into $G^b$ is called a \textbf{evolutionary history} $O_{G^a\rightarrow G^b}$.
Let $c(O_i)$ define the cost of the $i$-th operation in $O_{G^a\rightarrow G^b}$, then the cost of $O_{G^a\rightarrow G^b}$ is defined as $\sum_{i=1}^n c(O_i)$.
If there exists some evolutionary history $O_{G^a\rightarrow G^b}$ then $G^b$ is called a \textbf{potential ancestor} of $G^b$.
The cost $C(G^a \rightarrow G^b)$ to transform sequence $G^a$ into $G^b$ is then the minimal cost of all possible evolutionary histories transforming $G^a$ into $G^b$.

These definitions allow us to define the central problem in this work, the two species small phylogeny problem.
\begin{definition}{Two species small phylogeny problem}
\emph{
\begin{itemize}
\item Input:
 \begin{itemize}
  \item Two Genomes $G^a$ and $G^b$. 
  \item Set of allowed evolutionary operations $\mathcal O$
 \end{itemize}
\item Output:
 \begin{itemize}
  \item Potential common ancestor $G^*$ minimizing the cost $C(G^*\rightarrow G^a) + C(G^*\rightarrow G^b)$
 \end{itemize}
\end{itemize}
}
\end{definition}

In general the set of allowed evolutionary operations can include \textbf{Reversals} and \textbf{Transpositions} which change the genome organization, as well as \textbf{Losses}, \textbf{Insertions} and \textbf{Duplications} that modify the genome content.
The model proposed by Holloway \emph{et al.} only allows for the two operations \textbf{Loss} and \textbf{Duplication} defined as follows:
\begin{itemize}
\item A \textbf{Duplication} of size $k+1$ on genome $G = G_1\cdots G_n$ copies a substring $G[i, \ldots, i+k]$ (origin) to some location $j$ of $G$ outside the interval $[i, i+k]$ (target).  
\item A \textbf{Loss} of size $k+1$ removes a substring $G[i, \ldots, i+k]$ from $G$.
\end{itemize}

As both operations do not shuffle gene order Holloway \emph{et al.} suggest to pose the two-species small phylogeny problem for the duplication and loss model as an alignment problem.
Since the alignment of two extant genomes can only cover \textit{visible} evolutionary operations Holloway \emph{et al.} define a so called \textbf{visible history} and \textbf{visible ancestor}.
A visible history is an evolutionary history is defined as a triplet $(A, O_{A \rightarrow X}, O_{A \rightarrow Y})$ with $O_{A \rightarrow X}$  and $O_{A \rightarrow Y}$ being evolutionary histories with the following property:
For every duplication operation in $O_{A \rightarrow X}$  (resp. $O_{A \rightarrow Y}$) with origin $S$ and target $T$ it holds that this duplication is not followed by any other operation that will change the content in $S$ or $T$.
The genome $A$ is then called a \textbf{visible ancestor} of $X$ and $Y$.

In order to solve the two species small phylogeny problem as an alignment problem Holloway emph{et al.} first define the \textbf{labeling} of an alignment as follows:
\begin{definition}\label{def:labeling}{Labeling (of an alignment):\\}
Given an alignment of two genomes $G^a$ and $G^b$. The interpretation of this alignment as a sequence of losses and duplications is called a \textbf{labeling}. 
The cost of the labeling is the summed cost of all duplication and loss operations.
\end{definition}
Holloway \emph{et al.} show that there exist a one-to-one correspondence between a labeled alignments of two genomes $X$ and $Y$ and visible ancestors of $X$ and $Y$. 
Therefore in order to solve the two species small phylogeny problem for loss and duplication, in the first step they solve the so called Duplication-Loss Alignment Problem which they defined as: 
\begin{definition}\label{def:DLAProb}{Duplication-Loss Alignment Problem:\\}%
Given two Genomes $G^a$ and $G^b$, compute a \textbf{labeled} alignment of  $G^a$ and $G^b$ with minimum cost.
\end{definition}
Once this problem is solved the computation of the common ancestor that solves the small phylogeny problem is straight forward \cite{Holloway:2012}.

\subsection{Formal Problem Description}
When given two genomes $G^a$ of length $n$ and $G^b$ of length $m$, we first construct the alignment graph $T = (V,E)$, with $V = V^a \cup V^b$.
This graph is a complete bipartite graph containing two sets of nodes $V^a = v^a_1,\ldots,v^a_n$ and $V^b = v^b_1,\ldots,v^b_m$. 
Some node $v^a_i$ corresponds to the $i$-th gene ($G^a[i]$) in genome $G^a$ and some node $v^b_j$ corresponds to the $j$-th gene ($G^b[j]$) in genome $G^b$.
The set $E$ is the set of undirected edges, one for each pair of vertices $\{v^a_i, v^b_j \}$. 
Thus an undirected edge $\{v^a_i, v^b_j \}$ corresponds to the alignment of gene at position $i$ in $G^a$ and the gene at position $j$ in $G^b$.
The cost associated to the alignment of these two genes is $c_{ij}$.
If some alignment of $G^a$ and $G^b$ aligns gene $G^a[i]$ to $G^b[j]$, the corresponding edge $\{v^a_i, v^b_j \}$ is said to be realized by the alignment.
Depending on the selected scoring scheme not every pair of genes from $G^a$ and $G^b$ are allowed to be aligned, which would correspond to assigning a cost of $\infty$.
For the problem of genome alignment usually every gene of $G^a$ can only be aligned the same gene in  $G^b$ (with cost zero), such that only a very small subset of all alignment edges has a cost $<\infty$.
Thus we work on a sparse version of the alignment graph, where $E$ contains only those undirected edges $\{v^a_i, v^b_j \}$ where $c_{ij} < \infty$.

In a valid alignment of course only a subset of all alignment edges can be realized. 
Two alignment edges $e_1 = \{v^a_i, v^b_j \}$ and $e_2 = \{v^a_k, v^b_l \}$ are \textit{incompatible} if ($i \leq k$ and $l \leq j$) or ($i \geq k$ and $l \geq j$). 
We denote that by $\{e_1, e_2\} \in \mathcal{I}$, with $\mathcal{I}$ being the set of all pairs of \textit{incompatible} alignment edges.

Additionally we construct the sets of all possible duplications and losses for genome $G^a$ ($D^a$ and $L^a$) and genome $G^b$ ($D^b$ and  $L^b$).
For a duplication $d \in D^a$ with origin $G^a[i, \ldots, i+k]$ and target $G^a[j, \ldots, j+k]$ we define the functions $\textit{origin}(d) = [i, \ldots, i+k]$ and $\textit{target}(d) = [j, \ldots, j+k]$.
Similarly for some loss $l \in L^a$ that removes substring $G^a[i, \ldots, i+k]$ from $G^a$ we define the function $\textit{span}(l) = [i, \ldots, i+k]$.
According to the chosen scoring scheme every duplication operation $d \in D^a \cup D^b$ is charged some cost $c_d$.
The same holds for every loss operation $l \in L^a \cup L^b$ where we charge some cost $c_l$.

\begin{definition}\label{def:DupCycle}{Duplication-Cycle:\\}
A set of $k$ duplication events $D' \subseteq D^a$ ($D^b$ resp.) forms a duplication cycle iff there exists permutation $d_1, d_2,\ldots, d_k$ of the elements of $D'$ such that 
\begin{itemize}
\item[\emph{(a)}] $\textit{origin}(d_i) \cap \textit{target}(d_{i-1}) \neq \emptyset \quad \forall 2 \leq i \leq k$
\item[\emph{(b)}] $\textit{origin}(d_1) \cap \textit{target}(d_{k}) \neq \emptyset$ 
\end{itemize}
\end{definition}

\begin{proposition}
A valid labeling of an alignment does not contain a duplication cycle.
\end{proposition}
Intuitively given a set of duplications $d_1,\ldots, d_k$ such that for consecutive entries $d_i, d_{i+1}$ the regions $\textit{target}(d_i)$ and $\textit{origin}(d_{i+1})$ overlap.
A reasonable biological interpretation induces a strict chronological partial order to the duplication events. Thus for every duplication event $d_i, 1< i < k $ it must hold that $d_i$ happened after $d_{i-1}$ and before $d_{i+1}$ which we denote by $d_{i-1} < d_{i} < d_{i+1}$. 
This implies that no duplication cycle exists as it would contradict the strict partial order property of the duplication events and thus has no reasonable biological interpretation. 

\section{ILP formulation and valid inequalities}
For the remainder of this paper we will consider only losses of size 1 and simplify the notation such that for every node $v \in V^a$ the loss event $l_v \in L^a$ denotes the loss of the gene from $G^a$ corresponding to node $v$. 
The same holds for nodes in $V^b$.  
\subsection{An initial ILP model}
The formulation contains a \textbf{binary variable}:
\begin{itemize}
\item ${x_{ij}}$ for every alignment edge $e \in E, e=\{v^a_i, v^b_j \}$. 
\item ${ z_v}$ for every possible loss event $l_v \in L^a \cup L^b,$.
\item ${y_d}$ for every possible duplication $d \in D^a \cup D^b$.
\end{itemize}

In a valid solution to the duplication loss alignment problem every gene in Genome $G^a$ and $G^b$ must either be aligned to some gene in the other genome, labeled as a loss or labeled as the target of some duplication.
Additionally for both genomes there must not exist any duplication cycle.
For readability reasons we define the set $D^*$ as the set of all duplication cycles in $D^a$ and $D^b$, that is $D^* = \{D' \subseteq D^a \cup  D^b:D'$ forms a duplication cycle$\}$.
When we set all alignment edge costs to zero, the following ILP formulation equals the one by Holloway \emph{et al.} and solves the duplication-loss-alignment problem.

{\small
\begin{align} 
&\min\quad \sum_{\{v^a_i, v^b_j\}\in E} c_{ij} x_{ij} + \sum_{u \in V^a} c_u z_u + \sum_{v \in V^b} c_v z_v + \sum_{d \in D^a} c_d y_d  +  \sum_{d \in D^b} c_d y_d \qquad \qquad \label{eq:objFun}\\
&\text{w.r.t.} \nonumber
\end{align}

\begin{align} 
x_{ij} + x_{kl}  &\leq 1&& \quad   \forall \{\{v^a_i, v^b_j\}, \{v^a_k, v^b_l\}\} \in \mathcal{I} \label{eq:pairwise}\\
z_{v^a_i} + \sum_{\{v^a_i, v^b_j\} \in E}  x_{ij} + \sum_{\substack{d \in D^a\\ i \in \textit{target}(d)}}  y_d &= 1&& \quad   \forall 1 \leq i \leq n\label{eq:coverA}\\ 
z_{v^b_j} + \sum_{\{v^a_i, v^b_j\} \in E}  x_{ij} + \sum_{\substack{d \in D^b\\ j \in \textit{target}(d)}}  y_d &= 1&& \quad   \forall 1 \leq j \leq m \label{eq:coverB}\\
\sum_{d \in D'} \quad y_d &\leq |D'| - 1&& \quad  \forall D' \in D^* \label{eq:dupCycle}\\
x, y, z &\in \{0,1\} \label{eq:ilpEnd}&&
\end{align}
}

A solution to the ILP \eqref{eq:objFun}-\eqref{eq:ilpEnd} corresponds to a solution to the duplication-loss alignment problem.
But solving the ILP formulation above is not feasible for realistic values of $m$ and $n$ as the number of possible duplication cycles may grow exponentially with the length of the genomes.
Therefore instead of enumerating all inequalities of class \eqref{eq:dupCycle} beforehand, our approach and the one of Holloway \emph{et al.} is to first relax the ILP and drop the duplication cycle constraints \eqref{eq:dupCycle}.

Holloway \emph{et al.} solve the problem by iteratively solving the ILP formulation (without constraint \eqref{eq:dupCycle}) and searching for violated duplication cycle inequalities which are then added to the ILP which is re-solved.
These two steps are repeated until the solution of the ILP does not induce any duplication cycles.
In contrast, our cutting plane approach explained in Section~\ref{sec:branchncut} does not iteratively solve the ILP formulation.
Instead we search for violated duplication cycle inequalities at every node of the branch-and-cut tree of the ILP solver and add them as cutting planes.
Additionally we also identified other classes of valid inequalities that lead to a stronger LP relaxation of the ILP \eqref{eq:objFun}-\eqref{eq:ilpEnd} and thus can significantly speed up the solving process.
In Section~\ref{sec:validIneq} we define the classes of valid inequalities and in Section~\ref{sec:sepalgo} we show how to efficiently solve the separation problem for each class.
That is, given a (fractional) solution to the LP relaxation, identify a violated valid inequality.

\subsection{Valid Inequalities}\label{sec:validIneq}
In the following we let $\mathcal{P}$ be the convex hull of the feasible solutions to the above ILP.
To define valid inequalities for $\mathcal{P}$, we call pairs of alignments and/or duplications \emph{incompatible} if and only if a feasible solution cannot
assign a value of 1 to both of their corresponding variables. Incompatibility follows directly from constraints \eqref{eq:pairwise} or \eqref{eq:coverA} and \eqref{eq:coverB} in the ILP formulation. Then the 
\emph{incompatibility graph} $H$ has node set $E\cup D$ and an edge between all pairs of incompatible alignments and duplications. Similar to the multiple sequence
alignment approach in \cite{Althaus2005} we introduce maximal clique inequalities.

\subsubsection{Maximal clique inequalities}
Sets $K=K_E\cup K_D$ of pairwise incompatible alignments and duplications, with $K_E\subseteq E$ and $K_D\subseteq D^a\cup D^b$, correspond precisely to the cliques of the 
incompatibility graph $H$. If there is no alignment nor duplication that is incompatible with all alignments and duplications in $K$ the corresponding clique
is maximal. The following \emph{maximal clique inequality} is valid for $\mathcal{P}$:
\begin{equation}\label{eq:maxclique}
\sum_{e\in K_E} x_e + \sum_{d\in K_D} y_d \leq 1.
\end{equation}
Similarly to \cite{Althaus2005} we capture maximal sets of pairwise incompatible alignments by the following notation. We let 
$$\mathcal{E}(l_b\leftrightarrow l_e, m_b\leftrightarrow m_e)$$ denote the collection of all sets $S\subseteq E$ such that
\begin{enumerate}
 \item[(a)] all edges in $S$ are pairwise incompatible
 \item[(b)] for each edge $\{v^a_l,v^b_m\} \in S$, $l_b\leq l \leq l_e$ and $m_b\leq m \leq m_e$
 \item[(c)] $S$ is maximal with respect to properties (a) and (b).
\end{enumerate}
Furthermore, we let $$D^a(l\leftrightarrow m):=\{(v^a_p,v^a_q): p\leq l, q\geq m\}.$$ $D^b$ is defined analogously.
To show that maximal cliques in $H$ can be characterized by the following proposition, the same arguments as in \cite{Althaus2005} apply. 
\begin{proposition}\label{prop:maxClique}
 A clique $K=K_E\cup K_D$ in $H$ with sets $K_E\subseteq E$ and $K_D\subseteq D^a\cup D^b$ is maximal if and only if either 
 $$ K_E=\emptyset, K_D=D^x(\ell+1\leftrightarrow \ell) $$ for some $1\leq \ell \leq |G^x|$, $x\in\{a,b\}$, or $$ K_A\in\mathcal{E}(l_b \leftrightarrow l_e, 1\leftrightarrow|s^j|),
 K_D=D^x(l_b \leftrightarrow l_e)$$ for some $1\leq l_b\leq l_e\leq |G^x|$, $x\in\{a,b\}$.
\end{proposition}
In the next section we will show how this characterization can be exploited by an algorithm separating \eqref{eq:maxclique}.

\subsubsection{Duplication island inequalities}
In the following we consider duplications in $G^a$ and define $D:=D^a$ and $V:=V^a$. For duplications in $G^b$ the same holds.
Consider the graph $T'$ obtained by augmenting the alignment graph $T$ with directed arcs $A$ as follows. For every duplication $d\in D$ we add an arc
from the node representing the $i$th element in \textit{origin}(d) to the node representing the $i$th element in $\textit{target}(d)$,
for all $i=1\dots|\textit{origin}(d)|$. 
Furthermore, for every $(u,v)\in A$ we let $\mathcal{D}((u,v))$ be the set of duplications $d$ in $D$ such that there exists an $i$ with $u$ representing the $i$th element in $\textit{origin}(d)$ and $v$ representing the $i$th element in $\textit{target}(d)$.
Then for any set $S\subseteq V$, $\mathcal{D}(V\setminus S,S)$ denotes the set of duplications inducing arcs in the cut-set of $(V\setminus S, S)$, that is
$$ \mathcal{D}(V\setminus S,S) =\bigcup_{\substack{(u,v)\in A:\\ u\in V\setminus S, v\in S}} \mathcal{D}((u,v)).$$
\begin{theorem}
For every set $S\subseteq V$ the following inequality is valid for $\mathcal{P}$:
\begin{equation}\label{eq:dupisland}
\sum_{v\in S} z_v + \sum_{v\in S}\sum_{k=1}^{m}x_{\{v,v^b_k\}} + \sum_{d\in\mathcal{D}(V\setminus S,S)} y_d\geq 1  
\end{equation}
\end{theorem}
\begin{proof}
 Assume to the contrary that the sum on the left hand side of inequality~\eqref{eq:dupisland} is zero. Let graph $T''$ be obtained from graph $T'$ by 
 removing all alignment edges whose corresponding $x$-variable is $0$ and all arcs $(u,v)\in A$ with $y_d=0$ for all $d\in\mathcal{D}((u,v))$.
 Since every position in the genome must be covered (constraint~\eqref{eq:coverA}) and since $\sum_{v\in S} z_v + \sum_{v\in S}\sum_{k=1}^{m}x_{\{v,v^b_k\}}=0$,
 to every node $v\in S$ exactly one incoming arc in $A$ must be incident. As $\sum_{d\in\mathcal{D}(V\setminus S,S)}=0$ these arcs must originate at a 
 node in $S$. Thus, if we repeatedly traverse, starting at an arbitrary node in $S$, the unique incoming arc backwards, we will never leave node set $S$ and hence ultimately close a cycle. Due to constraint~\eqref{eq:dupCycle} the corresponding solution is infeasible.\qed 
\end{proof}

\subsubsection{Lifted duplication cycle inequalities}
Again, we consider duplications in $G^a$ and define $D:=D^a$ and $V:=V^a$. For duplications in $G^b$ the same holds.
In this section we introduce the \emph{lifted duplication cycle inequalities}, a class of constraints that dominate \eqref{eq:dupCycle}. 
The high-level idea this class of constraints is based on is similar to the one underlying the \emph{lifted mixed cycle inequalities} introduced in~\cite{Althaus2005}.
Consider a set of duplications
$C\subseteq D$, which is partitioned into sets $C^1,\dots,C^t$. If $C$ satisfies
\begin{enumerate}
 \item[(C1)] for $r=1,\dots,t$, all edges in $C^r$ are pairwise incompatible
 \item[(C2)] every set $\{d_1,\dots,d_t\}$, where $d_r$ is chosen arbitrarily from $C^r$ for $r=1,\dots,t$, forms a cycle according to Definition~\ref{def:DupCycle}
\end{enumerate}
then the inequality
\begin{equation}\label{eq:dupcycleLift}
\sum_{d\in C} y_d \leq t-1
\end{equation}
is valid for $\mathcal{P}$. Inequalities~\eqref{eq:dupCycle} are a special case of \eqref{eq:dupcycleLift} in which every set $C^r$ has cardinality one.
If additionally 
\begin{enumerate}
 \item[(C3)] $C$ is maximal with respect to properties (C1) and (C2)
\end{enumerate}
we call \eqref{eq:dupcycleLift} a \emph{lifted duplication cycle inequality}.
\begin{proposition}\label{theo:dupcycle}
 An inequality of the form \eqref{eq:dupcycleLift} with $C=\bigcup_{i=1}^t C^i$, $C\subseteq D^a$, is a lifted duplication cycle inequality if and only if there exists
a sequence of non-empty intervals $[a_1, b_1], [a_2, b_2],\dots [a_t, b_t]$ such that for $i=1,\dots,t$ it holds
\begin{enumerate}
 \item[\emph{(P1)}] $\bigcap_{d\in C^i} \textit{target}(d) = [a_{i+1},b_{i+1}]$
 \item[\emph{(P2)}] $\forall d\in C^i: \textit{origin}(d)\cap [a_i,b_i]\neq\emptyset$
 \item[\emph{(P3)}] $\forall d\in D\setminus C: \textit{target}(d)\cap [a_{i+1},b_{i+1}]\neq\emptyset\rightarrow$ 
 $$\textit{origin}(d)\cap [a_i,b_i]=\emptyset \vee \exists  d'\in C^{i+1}: \textit{target}(d)\cap\textit{origin}(d')=\emptyset$$
\end{enumerate}
 where $[a_{t+1},b_{t+1}]:=[a_1, b_1]$ and $C^{i+1}:=C^1$.
\end{proposition}
Intuitively, property (P1) captures condition (C1), property (P2) ensures that (C2) is satisfied, and (P3) implies maximality. A formal proof is given below.
Notice that condition (P3) is not equivalent to requiring $C$ to be maximal with respect to (P1) and (P2), since a duplication satisfying (P3) might intersect 
interval $[a_{i+1},b_{i+1}]$ only partially.
\begin{proof}
 To prove sufficiency assume set $C$ has the claimed structure (P1)-(P3). For $i=1,\dots,t$ any two duplications in $C^i$ contain at least one common vertex in their target violating constraint~\eqref{eq:coverA} and are thus incompatible. 
Furthermore, any set of duplications $\{d_{l_1},\dots,d_{l_t}\}$ with $d_{l_i}\in C^i$, $i=1,\dots,t$, forms a cycle according to Definition~\ref{def:DupCycle}, since due to properties (P1) and (P2)
$\textit{origin}(d_{i+1})\cap\textit{target}(d_i)\neq\emptyset$. 
Finally, assume $C$ is not maximal with respect to (C1) and (C2). 
Consider a duplication $d\notin C$ such that $C\cup\{d\}$ satisfies (C1) and (C2). In particular, there exists $1\leq r\leq t$ such that $d$ is incompatible with all
duplications in $C^r$ and thus $\textit{target}(d)\cap [a_{i+1},b_{i+1}]\neq\emptyset$. 
From condition (P3) it follows that either $\textit{origin}(d)\cap [a_i,b_i]=\emptyset$, in which case $d$ does not lie on a common cycle with any duplication from 
$C^{i-1}$, or $\exists  d'\in C^{i+1}: \textit{target}(d)\cap\textit{origin}(d')=\emptyset$, which implies that there exists a duplication $d'\in C^{i+1}$ such that
$d$ and $d'$ do not lie on a common cycle, violating in both cases condition (C2).

To prove necessity, we show that every set $C$ that satisfies (C1), (C2), and (C3), exhibits the claimed structure (P1)-(P3). For $i=1,\dots, t$, let
$[a_{i+1},b_{i+1}]:=\bigcap_{d\in C^i}\textit{target}(d)$, where $[a_{t+1},b_{t+1}]:=[a_1, b_1]$. 
Due to condition (C1), $[a_i,b_i]\neq\emptyset$ and thus property (P1) is satisfied. By the definition of cycles (C2) implies that the origin of every duplication in $C^i$ intersects the target of every duplication in $C^{i-1}$, $i=1,\dots, t$, where $C^0:=C^t$. For intervals $\textit{target}(d')$, $d'\in C^{i-1}$, with non-empty intersection this is equivalent to $\bigcap_{d'\in C^{i-1}}\textit{target}(d') \cap \textit{origin}(d)\neq \emptyset$, for all $d\in C^{i}$, which satisfies (P2).
Finally, assume (P3) does not hold, i.e. there exists a duplication $d\in D\setminus C$ with (i) $\textit{target}(d)\cap [a_{i+1},b_{i+1}]\neq\emptyset$, (ii)  
 $\textit{origin}(d)\cap [a_i,b_i]\neq\emptyset$,  and (iii) $\forall d'\in C^{i+1}: \textit{target}(d)\cap\textit{origin}(d')\neq\emptyset$. Due to constraint \eqref{eq:coverA},
 (i) causes $d$ to be incompatible with all duplications in $C^i$. Properties (i) and (ii) imply, by the definition of cycles (see Definition~\ref{def:DupCycle}),  
 that for every cycle $C$ that contains an arbitrary duplication
 $d'\in C^i$, replacing $d'$ by $d$ results in a cycle $C'$. Therefore $C\cup \{d\}$ satisfies (C1) and (C2), which is in contradiction to (C3).
\qed
\end{proof}

\section{A Branch-and-Cut Approach}\label{sec:branchncut}
In this section we show that the three classes of valid inequalities introduced in the previous section can be separated efficiently. At the end of the section we discuss further details of our 
implementation.

\subsection{Separation algorithms}\label{sec:sepalgo}
\begin{theorem}
For a given solution to the ILP for two genomes $G^a$ and $G^b$  the maximum weight maximal clique of alignment edges $$c^* = \arg \max_{c \in \mathcal{E}(l_b\leftrightarrow l_e, m_b\leftrightarrow m_e)} \sum_{v^a_i, v^b_j \in c} x_{ij}$$ within the interval $[l_b,l_e]$ in $G^a$ and  $[m_b,m_e]$ in $G^b$ can be computed in time $\mathcal{O}(l_{eb} m_{eb})$ with $l_{eb} = l_e-l_b + 1$ and $m_{eb} = m_e-m_b + 1$
\end{theorem}
\begin{proof}
In order to detect such a maximal clique we use the pair graph data structure which was introduced by Reinert \emph{et al.} \cite{Reinert97}
Given the subgraph $T'$ of the alignment graph induced by the vertex subsets $v^a_{l_b}, \ldots,v^a_{l_e}$ and $v^b_{m_b}, \ldots,v^b_{m_e}$ the corresponding pairgraph $PG(T')$ is a $l_{eb} \times m_{eb}$ directed grid graph where arcs go from bottom to top and right to left.
A node $n_{p,q}$ in in row $p$ and column $q$ of the pairgraph corresponds to the edge connecting node $v^a_{l_b+p-1}$ and $v^b_{m_b+q-1}$ in $T$.%
In the case of sparse alignment graph which is not a complete bipartite graph not every node in the pairgraph corresponds to an alignment edge in $E$.
Those that do correspond to some alignment edge are called essential nodes.
For every source to sink path $p= n_{1,m_{eb}},\ldots, n_{l_{eb},1}$ in the pairgraph, the edges of the alignment graph corresponding to essential nodes in $p$ form a maximal clique clique of conflicting alignment edges. %
In order to find the set $c^*$ we simply weight every essential node in the pairgraph by the value for the corresponding alignment edge variable in the actual solution and compute a longest node-weighted source to sink path.
Since the pairgraph is directed and acyclic and the number of vertices and arcs is $\mathcal{O}(l_{eb} m_{eb})$, the longest source to sink path can be computed in $\mathcal{O}(l_{eb} m_{eb})$ time.

\end{proof}

\begin{theorem}
 For a given point $(x^*,y^*,z^*)\in\mathbb{R}_+^{|E|+|D|+|V|}$, it can be determined in time $\mathcal{O}(n^3)$ whether a maximal clique inequality \eqref{eq:maxclique}
 is violated.
\end{theorem}

\begin{proof}
 We show how to separate maximal clique inequalities that involve duplications in $D:=D^a$; For cliques containing duplications in $D^b$ a symmetric argument applies.
 As suggested by the structure of maximal cliques (see Proposition~\ref{prop:maxClique}), and as described in \cite{Althaus2005}, we compute for all $1\leq l_b< l_e\leq n$ (a) 
 $K_E\in\mathcal{E}(l_b\leftrightarrow l_e,1\leftrightarrow m)$ that maximizes $\sum_{e\in K_E}x^*_e$ and (b) $\sum_{d\in D(l_b\leftrightarrow l_e)} y^*_d$.
 The corresponding maximal clique inequality is violated if $\sum_{e\in K_E}x^*_e + \sum_{d\in D(l_b\leftrightarrow l_e)} y^*_d > 1$.
 
 Concerning step (a), we compute for each of the $n-1$ possible values of $l_b$ the longest path tree from 
 $n_{1,m}$ in the pairgraph $PG(T')$, where $T'$ is the subgraph of the alignment graph induced by the corresponding sets of vertices $v^a_{l_b}, \ldots, v^a_{n}$ and $v^b_{1}, \ldots, v^b_{m}$. 
Computing the longest path tree takes time $\mathcal{O}(nm)$, and thus the total time required to execute step (a) is 
 $\mathcal{O}(n^2 m)$.
 
 Step (b) can be performed for all pairs of genes $i,j$ in time $\mathcal{O}(n^2)$ by the following dynamic program. First we define
 $\sigma_{i,j}:=\sum_{d\in D(i\leftrightarrow j)} y^*_d$ and $\pi_{i,j}:=\sum_{k=j}^{n}\sum_{d\in\mathcal{D}(i,k)}y^*_{d}$ and observe that $\sigma_{i,j}=\sigma_{i-1,j}+\pi_{i,j}$.
 First, for all $p=1,\dots,n$, we compute $\pi_{p,q}$, $q=p,\dots,n$, in the order $\pi_{p,n}=\sum_{d\in\mathcal{D}(p,n)}y^*_{d}$ and $\pi_{p,q}=\pi_{p,q+1}+
 \sum_{d\in\mathcal{D}(p,q)}y^*_{d}$
 in time $\mathcal{O}(n^2)$. Then we compute in the order $p=2,\dots,n$, $\sigma_{p,q}=\sigma_{p-1,q}+\pi_{p,q}$, $q=p,\dots,n$, which takes $\mathcal{O}(n^2)$
 time. \qed
 \obacht{Gesamtlaufzeit bzgl. n?}
\end{proof}

Next we will show that a slightly relaxed version of constraint \eqref{eq:dupisland} can be separated efficiently.  For that we define the multiplicity $\alpha(d,S)$ of a duplication $d$ in the cutset
of  a cut $(V\setminus S,S)$:
\begin{equation}\label{def:alpha}
\alpha(d,S):=|\{(u,v)\in A: u\in V\setminus S, v\in S \wedge d\in\mathcal{D}(u,v)\}|. 
\end{equation}
\begin{theorem} 
 For $x\in\{a,b\}$ let $D:=D^x$, $V:=V^x$, $n:=|V|$ and $m=|V^y|$, where $y$ is the complement of $x$ in $\{a,b\}$.
 For a given point $(x^*,y^*,z^*)\in\mathbb{R}_+^{|E|+|D|+|V|}$, it can be determined in time $\mathcal{O}(n^{3.5}\sqrt{|D|})$ whether the following relaxation of a duplication island constraint \eqref{eq:dupisland} is violated.
\begin{equation}\label{eq:relax_dupisland}
\sum_{v\in S} z^*_v + \sum_{v\in S}\sum_{k=1}^{m}x^*_{\{v,v^b_k\}} + \sum_{d\in\mathcal{D}(V\setminus S,S)} \alpha(d,S)\cdot y^*_d\geq 1 
\end{equation}
\end{theorem}
\begin{proof}
  For an arbitrary node $s\in V$ we let graph $G_s(V,A,w)$\obacht{confusion with V and A} contain a node $v_i$ for every gene $G^x[i]$ in genome $G^x$. Arc set $A=A_1\cup A_2$, where $A_1$ contains for every pair of vertices $(u,v) \in V\times V$ with
  $\mathcal{D}(u,v)\neq\emptyset$ $A_1$ an arc $(u,v)$ of weight 
  $w(u,v):=\sum_{d\in\mathcal{D}((u,v))}y^*_d$. $A_2$ contains for every $v\in V$ with $v\neq s$ an arc $(s,v)$ of weight 
  $w(s,v):=z^*_v+\sum_{k=1}^{m}x^*_{\{v,v^b_k\}}$.
  Then for every $S\subset V$ with $s\in V\setminus S$ the sum on the left hand side of inequality \eqref{eq:relax_dupisland} equals the weight of the cut $(V\setminus S, S)$ in $G_s$:
  \begin{align}
   \sum_{\substack{(u,v)\in A_1\cup A_2:\\ u\in V\setminus S,\, v \in S}} w(a) &= \sum_{\substack{(u,v)\in A_1:\\ u\in V\setminus S,\, v \in S}} w(u,v) + \sum_{(s,v)\in A_2: v\in S} w(s,v)\\
&= \sum_{\substack{(u,v)\in A_1:\\ u\in V\setminus S,\, v \in S}}\sum_{d\in\mathcal{D}((u,v))}y^*_d
    +\sum_{v\in S}\left( z^*_v+\sum_{k=1}^{m}x^*_{\{v,v^b_k\} }\right)\\
&= \sum_{d\in\mathcal{D}(V\setminus S,S)} \alpha(d,S)\cdot y^*_d
    +\sum_{v\in S} z^*_v + \sum_{v\in S}\sum_{k=1}^{m}x^*_{\{v,v^b_k\}}
  \end{align}
The last step follows directly from the definition of $\alpha(d,S)$, see \eqref{def:alpha}. Determining set $S^*$ that minimizes the left hand side of inequality 
\eqref{eq:relax_dupisland} is thus equivalent to computing the minimum $s-t$ cut in $G_s$, over all $s\in V$. This can be reduced to $2|V|-2$ maximum flow problems,
i.e. from an arbitrary node $s$ to all $t\neq s$ and from all $t\neq s$ to $s$, each taking time $\mathcal{O}(|V|^2\sqrt{|A|})$ using Goldberg-Tarjan's preflow
push-relabel algorithm.
\qed
\end{proof}

We next show how to separate a certain relaxation of the lifted duplication cycle constraints efficiently.
The high-level idea of the algorithm is to construct a graph, 
whose nodes represent elements that satisfy (P1) and whose edges connect intervals that
satisfy (P2). Similarly to the separation of lifted mixed cycles in the multiple sequence alignment problem \cite{Althaus2005}, a potentially violated constraint
in the relaxed form of a lifted duplication cycles is then obtained by a shortest path computation.

\begin{theorem}
For $x\in\{a,b\}$ let $D:=D^x$, $V:=V^x$, $n:=|V|$ and $m=|V^y|$, where $y$ is the complement of $x$ in $\{a,b\}$.
 For a given point $(x^*,y^*,z^*)\in\mathbb{R}_+^{|E|+|D|+|V|}$, it can be determined in time $\mathcal{O}(n^3+|D|n^2)$ whether the relaxation of a lifted duplication cycles \eqref{eq:dupcycleLift},
 in which for every interval $[a_i,b_i]$ in Proposition \ref{theo:dupcycle} $a_i=b_i$, for all $i=1,\dots,t$, is violated.
\end{theorem}

\begin{proof}
 We construct an arc weighted graph $G=(V,A,w)$ as follows. Similar to the alignment graph we have one node for every gene in the given genome. 
 For every pair of nodes $v_i$ and $v_j$ we compute the set of duplications $\mathcal{D}(i,j)$
 whose origin contain $v_i$ and whose target contain $v_j$, i.e. $\mathcal{D}(i,j):=\{d\in D: i\in\textit{origin}(d)\wedge j\in\textit{target}(d)\}$. For every non-empty set 
 $\mathcal{D}(i,j)$ we add an arc from node $v_{i}$ to node $v_{j}$. We define the weight of an arc as $$w((v_i,v_j)):=1-\sum_{d\in\mathcal{D}(i,j)}y^*_d.$$ The the violation of a lifted duplication
 cycle having the claimed structure, given by the sequence of nodes $v_{i_1}, v_{i_2},\dots,v_{i_t}$, with $v_{i_j}\in[a_j,b_j]$, is 
 $$\sum_{j=1}^t \sum_{d\in C^j} y^*_d - t + 1 = 1-\sum_{j=1}^t(1-\sum_{d\in C^j} y^*_d)=1-\sum_{j=1}^t w((v_{i_j},v_{i_{j+1}})),$$ 
 where $v_{i_{t+1}}:=v_{i_1}$. Note that sets $\mathcal{D}(i,j)$ satisfy (P1)-(P3) and thus the last inequality follows. The most violated lifted duplication cycle of the relaxed kind can therefore be
 obtained by computing the shortest arc-weighted path in $G$ from every node $v$ to itself (if it exists).

 Implemented na\"{\i}vely, the weight of the arcs in $A$ can be determined in $\mathcal{O}(|D|n^2)$. Note that due to constraint~\ref{eq:coverA} the arc weights are all non-negative and we can compute the shortest
 paths by Dijkstra's algorithm. Since graph $G$ has $\mathcal{O}(n^2)$ arcs and Dijkstra's algorithm is called $n$ times, the shortest cycle in $G$ can be found in time $\mathcal{O}(n^3)$.
\end{proof}

\section{NP-hardness}
The reduction is from the decision version of \textsc{Max-2SAT}, which is defined as follows. 
Given a boolean formula $\phi$ in conjunctive normal form with variables $x_1,\dots, x_n$, 
and clauses $C_1,\dots,C_m$, where each clause $C_i$ is a disjunction of exactly 2 literals, and a positive integer $k$. 
Decide whether there exists a truth assignment that satisfies at least $k$ clauses. It is well known that the decision version of
\textsc{Max-2SAT} is NP-complete.
  
We construct gadgets for each variable and each clause. We start with a description of the variable gadgets. 

\emph{Variable Gadget:} For a variable $x_i$ we let $m_i$ denote the number of clauses that the variable appears in. 
The gadget for variable $x_i$ consists of two strings $s^i_1$, $s^i_2$ of length $4m_i$ each. If $x_i$ appears in 
clauses $C_{i_1},\dots,C_{i_{m_i}}$ we set 
\begin{align*}
 s_1^i&=x_ic^i_{i_1}\cdots x_ic^i_{i_{m_i}}\bar{x}_i\bar{c}^i_{i_1}\cdots \bar{x}_i\bar{c}^i_{i_{m_i}}\\
 s_2^i&=\bar{x}_i\bar{c}^i_{i_1}\cdots \bar{x}_i\bar{c}^i_{i_{m_i}} x_i c^i_{i_1}\cdots x_i c^i_{i_{m_i}}
\end{align*}
\begin{lemma}\label{theo:vargadget}
 The optimal cost of an alignment of the two strings $s_1^i$, $s_2^i$, forming the variable gadget of a variable $x_i$ is $4m_i$.
 There exist exactly two optimal alignments that do not use duplications. 
\end{lemma}
\begin{proof}
 Neither of the two strings $s_1^i$, $s_2^i$, contains at least 2 consecutive characters that appear at least twice in the string. 
 Therefore there always exists an optimal solution that does not use any duplication. Such a solution is obtained by maximizing the
 number of matchings in the alignment. If any character $x_i$ or $c^i_{i_j}$ in $s_1^i$ is matched to any occurrence of that character
 in $s_2^i$, none of the characters $\bar{x}_i$ or $\bar{c}^i_{i_j}$ can be matched, and vice versa. Thus an alignment that matches all
 characters $x_i$ and all $c^i_{i_j}$, respectively all characters $\bar{x}_i$ and all $\bar{c}^i_{i_j}$, by aligning the $k$th character of
 $s_1^i[1\dots\ 2m_i]$ to the $k$th character of $s_2^i[2m_i+1\dots\ 4m_i]$, respectively the $k$th character of $s_1^i[2m_i+1\dots\ 4m_i]$
 to the $k$th character of $s_2^i[1\dots\ 2m_i]$, $1\leq k \leq 2m_i$, is optimal. Furthermore, an alignment that matches all characters $c^i_{i_j}$
 or all characters $\bar{c}^i_{i_j}$, $1\leq j \leq m_i$, only allows a unique matching of characters $x_i$, respectively $\bar{x}_i$.\qed
\end{proof}
An alignment that matches all characters $x_i$ and all $c^i_{i_j}$ of a variable gadget is said to be in FALSE configuration, and an 
alignment that matches all characters $\bar{x}_i$ and all $\bar{c}^i_{i_j}$ of a variable gadget is said to be in TRUE configuration.
Next we show that the variable gadgets can be independently set to a TRUE or FALSE configuration. 
\begin{lemma}
 The cost of an optimal alignment of strings 
\begin{align*}
X&=s_1^1\,s_1^2 \cdots s_1^n  \\
Y&=s_2^1\,s_2^2 \cdots s_2^n
 \end{align*}
is $8m$, where each variable gadget is in \emph{TRUE} or \emph{FALSE} configuration. 
\end{lemma}
\begin{proof}
 By Lemma \ref{theo:vargadget}, an alignment that set each variable gadget arbitrarily to a TRUE or FALSE configuration has an overall cost of 
 $\sum_{i=1}^n 4m_i=8m$. Furthermore, an optimal alignment of substrings $X$ and $Y$ is obtained by optimally aligning
 the substrings of the each variable gadget independently since the sets of characters that appear in different variable gadgets are disjoint. 
 By Lemma \ref{theo:vargadget} the claim follows.
\end{proof}

\emph{Clause Gadget:} The gadget for a clause $C_i = \ell_{i_1} \vee \ell_{i_2}$ is composed of two strings $t_1^i$, $t_2^i$, of length $4$ each. 
If $\ell_{i_1}$ is a variable $x_j$, we set $t_1^i[1\dots 2]=x_jc^j_i$, if $\ell_{i_1}$ is the negation of a variable $x_j$, i.e. $\bar{x}_j$, we
set $t_1^i[1\dots 2]=\bar{x}_j\bar{c}^j_i$. We define $t_1^i[3\dots 4]$ as a function of literal $\ell_{i_2}$ analogously. We set $t_2^i[1\dots 2]=t_1^i[3\dots 4]$
and $t_2^i[3\dots 4]=t_1^i[1\dots 2]$. As an example consider a clause $C_i$ of the form $x_j \vee \bar{x}_k$. Then
\begin{align*}
 t_1^i&=x_jc^j_i\bar{x}_k\bar{c}^k_i \\
 t_2^i&=\bar{x}_k\bar{c}^k_ix_jc^j_i
\end{align*}
Next we show a one-to-one correspondence between the optimal cost of a duplication-loss alignment instance that is composed of the variable gadgets and
a single clause gadget, and the evaluation of the clause under the implied truth assignment.
\begin{lemma}\label{theo:basecase}
 Consider the two strings 
 \begin{align*}
X&=s_1^1\,s_1^2 \cdots s_1^n\,\overbrace{\emph{\dollar} \cdots \emph{\dollar}}^5\, t_1^i  \\
Y&=s_2^1\,s_2^2 \cdots s_2^n\,\underbrace{\emph{\dollar} \cdots \emph{\dollar}}_5\, t_2^i
 \end{align*}
 obtained by concatenating all variable gadgets and the clause gadgets for a clause $C_i$, separated by string \emph{\dollar}\emph{\dollar}\emph{\dollar}\emph{\dollar}\emph{\dollar}.
 The cost of an optimal alignment of $X$ and $Y$ that sets all variable gadgets in \emph{TRUE} or \emph{FALSE} state is $8m$ if $C_i$ is satisfied under the truth 
 assignment implied by the variable gadgets and $8m+2$ otherwise. 
\end{lemma}
\begin{proof}
 Without loss of generality we assume that $x_j$ occurs positive, and $x_k$ occurs negative in $C_i$, i.e. $C_i=x_j\vee\bar{x}_k$. The other 3 cases can be covered
 analogously \ignore{and are summarized in Table \ref{}}. Consider strings $X$ and $Y$.
 No two characters of $t_1^i$ and $t_2^i$ can be matched simultaneously to a character in $s_2^1\,s_2^2 \cdots s_2^n$, respectively $s_1^1\,s_1^2 \cdots s_1^n$,
 since the alignment edges would cross. Therefore, at most $4$ characters in the clause gadget can be matched to characters in the variable gadgets. 
 In this case none of the characters $\dollar $ can be matched. Replacing the at most 4 matchings by duplications and losses will increases the cost by at most $8$,
 while matching all characters $\dollar $ decreases the cost by $10$. Thus an optimal alignment will not match any character in the clause gadget 
 to a character in one of the variable gadgets.
 
 If $C_i$ is not satisfied, i.e. the variable gadget for $x_j$ is in FALSE configuration and the variable gadget for $x_k$ is in TRUE configuration, 
 no two consecutive characters of $t_1^i$ ($t_2^i$) can be duplicated, since the characters in any of their occurrences are matched in the variable gadgets.
 On the other hand, matching characters $x_jc^j_i$ or $\bar{x}_k\bar{c}^k_i$ in the clause gadget and covering the remaining characters by duplications 
 from the corresponding substrings in the variable gadgets causes an additional cost of 2. Since at most 2 characters can be matched in a clause gadget and
 no 3 or more consecutive characters in $t_1^i$ ($t_2^i$) occur in $s_1^1\,s_1^2 \cdots s_1^n$ ($s_2^1\,s_2^2 \cdots s_2^n$), this is optimal.
 
 If both literals of $C_i$ evaluate to TRUE, i.e. the variable gadget for $x_j$ is in TRUE configuration and the variable gadget for $x_k$ is in FALSE
 configuration, both $x_jc^j_i$ and $\bar{x}_k\bar{c}^k_i$ are unmatched in both $s_1^1\,s_1^2 \cdots s_1^n$ and $s_2^1\,s_2^2 \cdots s_2^n$. Characters $x_jc^j_i$ in a variable gadget 
 can be a duplication of the corresponding characters in the clause gadget and vice versa, but one duplication invalidates the reverse. The same holds for
 $\bar{x}_k\bar{c}^k_i$. Furthermore, reverse duplications contribute equally to the total cost of the solution. Thus if we cover all characters in the clause
 gadgets by duplications we incur an additional cost of $4$. However, if we arbitrarily
 choose to match $x_jc^j_i$ or $\bar{x}_k\bar{c}^k_i$ in the clause gadget, the same characters in the variable gadgets can be the product of a duplication
 and the total cost reduces by $4$ to $8m+4-4$.

 If one literal of $C_i$ evaluates to TRUE and the other to FALSE, the argument is the same as in the previous case, except that instead of choosing
 the characters to match in the clause gadget arbitrarily, we match the characters whose corresponding literal evaluates to TRUE.\qed
\end{proof}
Finally we construct an instance to the duplication-loss alignment problem by concatenating all variable and clause gadgets, separated in the following way:
\begin{align*}
X &= s_1^1\,s_1^2 \cdots s_1^n\,\overbrace{\dollar_1 \cdots \dollar_1}^5\, t_1^1\,\overbrace{\dollar_2 \cdots \dollar_2}^5 \,t_1^2\,\overbrace{\dollar_3 \cdots \dollar_3}^5 \cdots \overbrace{\dollar_m \cdots \dollar_m}^5 \,t_1^m\\
Y &= s_2^1\,s_2^2 \cdots s_2^n\,\underbrace{\dollar_1 \cdots \dollar_1}_5\, t_2^1\,\underbrace{\dollar_2 \cdots \dollar_2}_5 \,t_2^2\,\underbrace{\dollar_3 \cdots \dollar_3}_5\cdots \underbrace{\dollar_m \cdots \dollar_m}_5 \,t_2^m 
\end{align*}

\begin{lemma}\label{theo:mainlemma}
 Consider the two strings 
 \begin{align*}
 X &= s_1^1\,s_1^2 \cdots s_1^n\,\overbrace{\emph{\dollar}_1 \cdots \emph{\dollar}_1}^5\, t_1^1\,\overbrace{\emph{\dollar}_2 \cdots \emph{\dollar}_2}^5 \,t_1^2\,\overbrace{\emph{\dollar}_3 \cdots \emph{\dollar}_3}^5 \cdots \overbrace{\emph{\dollar}_m \cdots \emph{\dollar}_m}^5 \,t_1^m\\
Y &= s_2^1\,s_2^2 \cdots s_2^n\,\underbrace{\emph{\dollar}_1 \cdots \emph{\dollar}_1}_5\, t_2^1\,\underbrace{\emph{\dollar}_2 \cdots \emph{\dollar}_2}_5 \,t_2^2\,\underbrace{\emph{\dollar}_3 \cdots \emph{\dollar}_3}_5\cdots \underbrace{\emph{\dollar}_m \cdots \emph{\dollar}_m}_5 \,t_2^m 
 \end{align*}
The cost of an alignment of $X$ and $Y$ that sets all variables gadgets in \emph{TRUE} or \emph{FALSE} state is $10m-2k$, where $k$ is the number of clauses satisfied
under the implied truth assignment.
\end{lemma}
\begin{proof}
 The proof is by induction on the number of clause gadgets $q$. We claim that the optimal cost of an alignment of $X$ and $Y$ restricted to the left-most $q$ clause
 gadgets has cost $8m+2(q-k)$, where $k$ is the number of clause gadgets among the $q$ left-most clause gadgets whose corresponding clause is satisfied under the 
 implied truth assignment.
 We also show that in an optimal solution no character of any variable gadget is matched to a character of any clause gadget. 
 The base case ($q=1$) holds by Lemma \ref{theo:basecase} and the construction in the proof of the same lemma. To show the induction step, assume that the claim
 holds for $q=\ell$. To show that the claim holds for $q=\ell+1$, we observe that 
 no two characters of $t_1^{\ell+1}$ and $t_2^{\ell+1}$ can be matched simultaneously to a character in $Y^{\ell}$, respectively
 $X^{\ell}$,
 since the alignment edges would cross. Therefore, at most $4$ characters in the clause gadget for variable $x_{\ell+1}$ can be matched to characters in strings 
 $X^{\ell}$ and $Y^{\ell}$.
 In this case none of the characters $\dollar_{\ell+1} $ can be matched, since the only occurrence of characters $\dollar_{\ell+1} $ is directly preceding $t_1^{\ell+1}$ and $t_2^{\ell+1}$.
 Replacing the at most 4 matchings by duplications and losses will increase the cost by
 at most $8$,
 while matching all characters $\dollar_{\ell+1} $ decreases the cost by $10$. Thus an optimal alignment will not match any character in the clause gadget 
 to a character in $X^{\ell}$, respectively $Y^{\ell}$. As no two consecutive characters in $t_1^{\ell+1}$ or $t_2^{\ell+1}$ appear in any other clause gadget, 
 there always exists an optimal solution that does not contain any duplication between the gadget for clause $C_{\ell+1}$ and any other clause gadget.
 Furthermore, substrings $t_1^{\ell+1}[1\dots 2]$, $t_1^{\ell+1}[3\dots 4]$, and $t_2^{\ell+1}[1\dots 2]$, $t_2^{\ell+1}[3\dots 4]$ appear exactly 
 once in $X^{\ell}$, respectively $Y^{\ell}$, and do not intersect any occurrence of a sequence of at least two characters appearing multiple times in 
 $X^{\ell}$, respectively $Y^{\ell}$. Therefore, and due to the structural assumption of the induction hypothesis, in an optimal alignment of $X^{\ell}$ and
 $Y^{\ell}$, the unmatched characters in the unique occurrences of these substrings in $X^{\ell}$ and $Y^{\ell}$ will be covered by losses. Therefore, by the same
 arguments as in the proof of Lemma~\ref{theo:basecase}, an optimal alignment of $X^{\ell+1}$ and $Y^{\ell+1}$ incurs no additional cost compared to an optimal alignment of 
 $X^{\ell}$ and $Y^{\ell}$ if clause $C_{\ell+1}$ is satisfied under the implied truth assignment, and an additional cost of 2 otherwise, 
 summing to an overall cost of $8m+2(\ell-k)=8m+2((\ell+1)-(k+1))$, respectively $8m+2(\ell-k)+2=8m+2((\ell+1)-k)$. \qed
\end{proof}

\begin{theorem}
 The duplication-loss alignment problem is NP-hard.
\end{theorem}
\begin{proof}
 The proof follows from Lemma \ref{theo:mainlemma} and the fact that the decision version of \textsc{Max-2SAT} is $NP$-hard.
\end{proof}

\section{Experimental results}
\label{sec:exp}
In this section we present the preliminary results of the comparison between the branch-and-cut algorithm as outlined in Section\ref{sec:branchncut} and the iterative ILP formulation suggested by Holloway \emph{et al.} in terms of run time.
We implemented both approaches in C++ and used the \textsc{cplex} solver version 12.4 as ILP solver.
All experiments were run single threaded on a 2.67~GHz Intel Xeon cpu.

For the implementation of the graphs we used the lemon graph library \cite{lemon11} and the seqan library \cite{seqan08} that provide the standard graph algorithms (max-flow-min-cut, dag shortest path, Dijkstra). 
To detect all duplication cycles induced by an intermediate solution in the iterative ILP approach we first construct a digraph with a node $v_i$ for every duplication event $d_i \in D^a$ (resp $D^b$).
We insert a directed edge $(v_i, v_j)$ if \textit{origin($d_j$)} intersects \textit{target($d_i$)}.
Then we weight every edge $(v_i, v_j)$ with the value $(1-y_i)$ where $y_i$ is the duplication variable for $d_i$.
In this directed graph we then enumerate all cycles with weight strictly less than 1 (same argument as for the lifted cycle separation).
The cycles are enumerated using a slight variant of the DFS-based algorithm by Johnson~\cite{Johnson75}.
For every detected cycle the corresponding violated duplication cycle constraint is added to the ILP before it is solved again.

For the branch-and-cut approach we utilized the user cut interface shipped with the \textsc{ilog  cplex}/Concert library.
For both algorithms we used  the default solver settings and measured the run time to compute an optimal solution.

We used the same scoring scheme as Holloway \emph{et al.} where alignments of homologous genes have cost 0, while every single gene loss and every duplication event is charged a cost of $1$.
Obviously there may exist multiple optimal solutions for some instances therefore both algorithms not necessarily report the same solution.

For the benchmark we used two types of data, real world data and simulated data.

\subsubsection{Real-world instances} 
We compared the two approaches on two sets of real-world instances that were also used in~\cite{Holloway:2012}.
The sets contain the stable tRNA and rRNA contents of 12 \textit{Bacillus} and 6 \textit{Vibrionaceae} lineages that were preprocessed like discussed in~\cite{Holloway:2012}.
So they are linearized according to their origin of replication and inverted segments are manually re-inverted.
For both sets we ran both algorithms for all pairs of genomes leading to 66 pairs for \textit{Bacillus} and 15 pairs for \textit{Vibrionaceae}.
The average run time of the direct iterative ILP algorithm on the \textit{Bacillus} instances was around 19 seconds, while the our branch-and-cut algorithm took less than 1.5 seconds.

For the \textit{Vibrionaceae} pairs, the advantage of our algorithm is even more prominent as the ILP  did not finish after several days on some instances, while the branch-and-cut algorithm always needed less than one hour - on most instances only a few minutes.
A more detailed benchmark on this dataset is still in process.
\subsubsection{Simulated instances}
The simulation of input instances follows the strategy of Holloway \emph{et al.}
The simulation is performed in the following steps.
First a random sequence $R$ of length $n$ and alphabet size $a$ is simulated where the alphabet symbols at each position are iid.
In the second step $l$ moves (single gene loss or duplication event) are applied to $R$ where the length of a duplication follows a Gaussian distribution with mean 5 and standard deviation 2 and the start position of every move is uniformly distributed.
This sequence is then used as the ancestor genome $X$ and two extant genomes are generated by again applying $l$ moves to $X$ for each of them.
In Table~\ref{tab:simresults} we present run time results for several settings of parameters $n$, $l$, and $a$ where we simulated 50 instances for each setting.

\begin{table}[htdp] 
\setlength{\tabcolsep}{15pt}
\caption{Benchmark on simulated data}
\begin{center}
\begin{tabular}{|c||c|c|}
\hline
Setting & \multicolumn{2} {|c|} {avg. run time in sec.} \\
\cline{2-3} 
$(n,l,\alpha)$ & branch-and-cut &  iterative ILP\\
\hline
\hline
(100,10,50) & 0.3  & 8.9 \\
\hline
(200,20,100) & 1.5  &  149.4\\
\hline
(400,40,20) & 7.0  &  2499.8\\
\hline
\end{tabular}\\
(see text for details)
\end{center}
\label{default}
\label{tab:simresults}
\end{table}%

\section{Conclusion}
\label{sec:discussion}
The preliminary results from the run time comparison show that the branch-and-cut algorithm clearly outperforms the ILP. In particular for larger instances (bigger $n$) and pairs of rather distant genomes like in the \textit{Vibrionaceae} dataset, the improvement in terms of run time is immense.
Therefore the branch-and-cut algorithm allows to solve more difficult instances than the pairs of \textit{Bacillus}  genomes on a desktop pc and does not require compute clusters to solve the instances in a reasonable amount of time.

\subsubsection{Acknowledgments} 
We thank the authors of \cite{Holloway:2012} for providing us with the genome datasets and software. 

\section*{Author contributions}
Stefan Canzar and Sandro Andreotti are joint first authors

\bibliographystyle{ieeetr}
\bibliography{genaln}

\end{document}